\documentclass[a4paper,10pt,openright]{article}

\usepackage[english]{babel}
\usepackage[T1]{fontenc}
\usepackage{indentfirst}
\usepackage[utf8]{inputenc}
\usepackage{amsmath}
\usepackage{amsfonts}
\usepackage{amssymb}
\usepackage{mathrsfs}
\usepackage{amsthm}
\usepackage{xfrac}
\usepackage{lmodern}
\usepackage{fix-cm}
\usepackage{listings}
\usepackage{fancyvrb}
\usepackage{enumerate}   

\newcommand{\ZZ}{\mathbb{Z}}

\newcommand{\HS}{\mathrm{HS}}
\newcommand{\cone}{\mathrm{cone}}
\newcommand{\CMtype}{\text{CM-type}}
\newcommand{\lc}{\mathrm{lc}}
\newcommand{\Ebeta}{\widehat{\beta}}

\newcommand{\Tor}{\mathrm{Tor}}

\let\To=\longrightarrow

\def\opn#1#2{\def#1{\operatorname{#2}}} 
\opn\chara{char} \opn\length{\ell} \opn\pd{pd} \opn\rk{rk}
\opn\projdim{proj\,dim} \opn\injdim{inj\,dim} \opn\rank{rank}
\opn\depth{depth} \opn\grade{grade} \opn\height{height}
\opn\embdim{emb\,dim} \opn\codim{codim}

\opn\Tr{Tr} \opn\bigrank{big\,rank}
\opn\superheight{superheight}\opn\lcm{lcm}
\opn\trdeg{tr\,deg}
\opn\reg{reg} \opn\lreg{lreg} \opn\ini{in} \opn\lpd{lpd}
\opn\size{size}\opn\bigsize{bigsize}
\opn\cosize{cosize}\opn\bigcosize{bigcosize}
\opn\sdepth{sdepth}\opn\sreg{sreg}
\opn\link{link}\opn\fdepth{fdepth}

\newtheoremstyle{theorem}
{12pt} 
{12pt} 
{\slshape 

} 
{} 
{\bfseries} 
{} 
{ } 
{} 

\newtheoremstyle{definition}
{12pt} 
{12pt} 
{\slshape 

} 
{} 
{\bfseries} 
{} 
{ } 
{} 

\newtheoremstyle{algorithm}
{10pt} 
{10pt} 
{\slshape 

} 
{} 
{\bfseries} 
{} 
{ } 
{} 

\theoremstyle{theorem}
\newtheorem{theorem}{Theorem}
\newtheorem{corollary}[theorem]{Corollary}
\newtheorem{proposition}[theorem]{Proposition}
\newtheorem{lemma}[theorem]{Lemma}

\newtheorem{setup}[theorem]{Set-up}

\theoremstyle{definition}

\theoremstyle{remark}
\newtheorem{remark}[theorem]{Remark}
\newtheorem{question}[theorem]{Question}

\theoremstyle{algorithm}

\numberwithin{theorem}{section}

\title{Extremal Betti numbers of some Cohen-Macaulay binomial edge ideals}
\author{Carla Mascia, Giancarlo Rinaldo \footnote{Email addresses: carla.mascia@unitn.it, giancarlo.rinaldo@unitn.it} \\ University of Trento}

\begin{document}
\maketitle

\begin{abstract}
We provide the regularity and the Cohen-Macaulay type of binomial edge ideals of Cohen-Macaulay cones, and we show the extremal Betti numbers of some classes of Cohen-Macaulay binomial edge ideals: Cohen-Macaulay bipartite and fan graphs. In addition, we compute the Hilbert-Poincaré series of the binomial edge ideals of some Cohen-Macaulay bipartite graphs. 
\end{abstract}

\section*{Introduction}
Binomial edge ideals were introduced in 2010 independently by Herzog et al. in \cite{HHHKR} and  by Ohtani in \cite{MO}. They are a natural generalization
of the ideals of 2-minors of a $2\times n$-generic matrix: their generators are those 2-minors whose column indices correspond to the edges of a graph. More precisely, given a simple graph $G$ on $[n]$ and the polynomial ring $S=K[x_1,\dots,x_n,y_1, \dots,y_n]$ with $2n$ variables over a field $K$, the \textit{binomial edge ideal} associated to $G$ is the ideal $J_G$ in $S$ generated by all the binomials $\{x_iy_j -x_jy_i | i,j \in V(G) \text{ and } \{i,j\} \in E(G)\}$, where $V(G)$ denotes the vertex set and $E(G)$ the edge set of $G$. 
Many algebraic and homological properties of these ideals have been investigated, such as the Castelnuovo-Mumford regularity and the projective dimension, see for instance \cite{HHHKR}, \cite{EHH}, \cite{MK}, \cite{KM}, and \cite{RR}. 
Important invariants which are provided by the graded finite free resolution are the extremal Betti numbers of $J_G$.
Let $M$ be a finitely graded $S$-module. Recall the Betti number $\beta_{i,i+j}(M) \neq 0$ is called \textit{extremal} if $\beta_{k,k+\ell} =0$ for all pairs $(k,\ell) \neq (i,j)$, with $k \geq i, \ell \geq j$. A nice property of the extremal Betti numbers is that $M$ has an unique extremal Betti number if and only if $\beta_{p,p+r}(M) \neq 0$, where $p = \projdim M $ and $r = \reg M$. In this years, extremal Betti numbers were studied by different authors, also motivated by Ene, Hibi, and Herzog's conjecture (\cite{EHH}, \cite{HR}) on the equality of the extremal Betti numbers of $J_G$ and $\mathrm{in}_<(J_G)$. Some works in this direction are \cite{B}, \cite{DH}, and \cite{D}, but the question has been completely and positively solved by Conca and Varbaro in \cite{CV}. The extremal Betti numbers of $J_G$ are explicitly provided by Dokuyucu, in \cite{D}, when $G$ is a cycle or a complete bipartite graph, by Hoang, in \cite{H}, for some closed graphs, and by Herzog and Rinaldo, in \cite{HR}, and Mascia and Rinaldo, in \cite{MR}, when $G$ is a block graph. In this paper, we show the extremal Betti numbers for binomial edge ideals of some classes of Cohen-Macaualy graphs: cones, bipartite and fan graphs. The former were introduced and investigated by Rauf and the second author in \cite{RR}. They construct Cohen-Macaulay graphs by means of the formation of cones: connecting all the vertices of two disjoint Cohen-Macaulay graphs to a new vertex, the resulting graph is Cohen-Macaulay. For these graphs, we give the regularity  and also the Cohen-Macaulay type (see Section \ref{Sec: cones}) . The latter two are studied by Bolognini, Macchia and Strazzanti in \cite{BMS}. They classify the bipartite graphs whose binomial edge ideal is Cohen-Macaulay. In particular, they present a family of bipartite graphs $F_m$ whose binomial edge ideal is Cohen-Macaulay, and they prove that, if $G$ is connected and bipartite, then $J_G$ is Cohen-Macaulay if and only if $G$ can be obtained recursively by gluing a finite number of graphs of the form $F_m$ via two operations. In the same article, they describe a new family of Cohen-Macaulay binomial edge ideals associated with non-bipartite graphs, the fan graphs. For both these families, in \cite{JK}, Jayanthan and Kumar compute a precise expression for the regularity, whereas in this work we provide the unique extremal Betti number of the binomial edge ideal of these graphs (see Section \ref{Sec: bipartite and fan graphs} and Section \ref{Sec: Cohen-Macaulay bipartite graphs}). In addition, we exploit the unique extremal Betti number of $J_{F_m}$ to describe completely its Hilbert-Poincaré series (see Section  \ref{Sec: bipartite and fan graphs}). 

\section{Betti numbers of binomial edge ideals of disjoint graphs}\label{Sec: preliminaries}

In this section we recall some concepts and notation on graphs that we will use in the article.

Let $G$ be a simple graph with vertex set $V(G)$ and edge set $E(G)$. A subset $C$ of $V(G)$ is called a \textit{clique} of $G$ if for all $i$ and $j$ belonging to $C$ with $i \neq j$ one has $\{i, j\} \in E(G)$. The \textit{clique complex} $\Delta(G)$ of $G$ is the simplicial complex of all its cliques. A clique $C$ of $G$ is called \textit{face} of $\Delta(G)$ and its \textit{dimension} is given by $|C| -1$.  A vertex of $G$ is called {\em free vertex} of $G$ if it belongs to exactly one maximal clique of $G$. A vertex of $G$ of degree 1 is called \textit{leaf} of $G$. A vertex of $G$ is called a \textit{cutpoint} if the removal of the vertex increases the number of connected components. A graph $G$ is {\em decomposable}, if there exist two subgraphs $G_1$ and $G_2$ of $G$,  and a decomposition $G=G_1\cup G_2$ with $\{v\}=V(G_1)\cap V(G_2)$, where $v$ is a free vertex of $G_1$ and $G_2$.

\begin{setup}\label{setup} Let $G$ be a graph on $[n]$ and $u \in V(G)$ a cutpoint of $G$. We denote by
\begin{align*}
& G' \text{ the graph obtained from } G \text{ by connecting all the vertices adjacent to } u, \\
& G'' \text{ the graph obtained from } G \text{ by removing } u, \\
& H \text{ the graph obtained from } G'  \text{ by removing } u.
\end{align*}

\end{setup}
Using the notation introduced in Set-up \ref{setup}, we consider the following short exact sequence
\begin{equation}\label{Exact}
 0\To S/J_G \To S/J_{G'}\oplus S/((x_u, y_u)+J_{G''})\To S/((x_u,y_u)+J_{H})  \To 0 
\end{equation}
For more details, see Proposition 1.4, Corollary 1.5 and Example 1.6 of \cite{R}. 
From (\ref{Exact}), we get the following long exact sequence of Tor modules
\begin{equation}\label{longexact}
\begin{aligned}
&\cdots\rightarrow T_{i+1,i+1+(j-1)}(S/((x_u,y_u)+J_H)) \rightarrow T_{i,i+j}(S/J_G) \rightarrow  \\
& \hspace{-0.4cm}T_{i,i+j}(S/J_{G'}) \oplus T_{i,i+j}(S/((x_u, y_u)+J_{G''})) \rightarrow T_{i,i+j}(S/((x_u,y_u)+J_H)) \rightarrow 
\end{aligned}
\end{equation}
where  $T_{i,i+j}^S(M)$ stands for $\Tor_{i,i+j}^S(M,K)$ for any $S$-module $M$, and $S$ is omitted if it is clear from the context. 

\

\begin{lemma}\label{Lemma: beta_p,p+1 e beta_p,p+2}
Let $G$ be a connected graph on $[n]$. Suppose $J_G$ be Cohen-Macaulay, and let $p=\projdim S/J_G$. Then 
\begin{enumerate}
\item[(i)] $\beta_{p,p+1} (S/J_G) \neq 0$ if and only if $G$ is a complete graph on $[n]$.
\item[(ii)] If $G= H_1 \sqcup H_2$, where $H_1$ and $H_2$ are graphs on disjoint vertex sets, then $\beta_{p,p+2} (S/J_{G}) \neq 0$  if and only if $H_1$ and $H_2$ are complete graphs.
\end{enumerate}
\end{lemma}

\begin{proof}
(i) In \cite{HMK}, the authors prove that for any simple graph $G$ on $[n]$, it holds
\begin{equation}\label{eq: linear strand}
\beta_{i,i+1} (S/J_G) = i f_{i}(\Delta(G))
\end{equation}
where $\Delta(G)$ is the clique complex of G and $f_{i}(\Delta(G))$ is the number of faces of $\Delta(G)$ of dimension $i$. Since $J_G$ is Cohen-Macaulay, it holds $p=n-1$, and the statement is an immediate consequence of Equation (\ref{eq: linear strand}), with $i=p$.\\

\noindent (ii) Since $J_G$ is generated by homogeneous binomials of degree 2, $\beta_{1,1}(S/J_G) = 0$. This implies that $\beta_{i,i}(S/J_G) = 0$ for all $i \geq 1$. For all $j\geq 1$, we have

\begin{equation*}
\beta_{p,p+j} (S/J_G) = \sum_{\substack{1 \leq j_1, j_2 \leq r \\ j_1+j_2 = j}} \beta_{p_1, p_1+j_1}(S_1/J_{H_1})\beta_{p_2, p_2+j_2}(S_2/J_{H_2})
\end{equation*}
For $j=2$, we get 
\begin{equation}\label{eq on beta_p,p+2 disjoint graphs}
\beta_{p,p+2} (S/J_G) = \beta_{p_1, p_1+1}(S_1/J_{H_1})\beta_{p_2, p_2+1}(S_2/J_{H_2}).
\end{equation}
By part (i), both the Betti numbers on the right are non-zero if and only if $H_1$ and $H_2$ are complete graphs, and the thesis follows. 
\end{proof}

\ \\

Let $M$ be a finitely graded $S$-module. Recall the Cohen-Macaulay type of $M$, that we denote by $\CMtype(M)$, is $\beta_p(M)$, that is the sum of all $\beta_{p,p+i}(M)$, for $i=0,\dots,r$, where $p= \projdim M$, and $r=\reg M$.  When $S/J_G$ has an unique extremal Betti number, we denote it by $\Ebeta(S/J_G)$.

\begin{lemma}\label{Lemma: Betti numbers of disjoint graphs}
Let $H_1$ and $H_2$ be connected graphs on disjoint vertex sets and $G=H_1\sqcup H_2$. Suppose $J_{H_1}$ and $J_{H_2}$ be Cohen-Macaulay binomial edge ideals. Let $S_i = K[\{x_j,y_j\}_{j \in V(H_i)}]$ for $i=1,2$. Then 
\begin{enumerate}
\item[(i)] $\CMtype (S/J_{G}) =  \CMtype(S_1/J_{H_1})  \CMtype(S_2/J_{H_2})$.
\item[(ii)] $\Ebeta(S/J_G) =  \Ebeta(S_1/J_{H_1})\Ebeta(S_2/J_{H_2})$.
\end{enumerate}

\end{lemma}

\begin{proof}
(i) The equality $J_{G} = J_{H_1} + J_{H_2}$ implies that the minimal graded free resolution of $S/J_G$ is the tensor product of the minimal graded free resolutions of $S_1/J_{H_1}$ and $S_2/J_{H_2}$, where $S_i = K[\{x_j,y_j\}_{j \in V(H_i)}]$ for $i=1,2$. Then 
\[
\beta_t(S/J_G) = \sum_{k=0}^t \beta_k(S_1/J_{H_1})\beta_{t-k}(S_2/J_{H_2}).
\]
Let $p = \projdim S/J_G$, that is $p=p_1+p_2$, where $p_i = \projdim S_i/J_{H_i}$ for $i=1,2$. Since $\beta_k (S_1/J_{H_1}) = 0$ for all $k > p_1$ and $\beta_{p-k}(S_2/J_{H_2})=0$ for all $k < p_1$, it follows
\[
\beta_p(S/J_G) = \beta_{p_1}(S_1/J_{H_1})\beta_{p_2}(S_2/J_{H_2}).
\]

\noindent (ii) Let $r= \reg S/J_G$. Consider 
\begin{equation*}
\beta_{p,p+r} (S/J_G) = \sum_{\substack{1 \leq j_1, j_2 \leq r \\ j_1+j_2 = r}} \beta_{p_1, p_1+j_1}(S_1/J_{H_1})\beta_{p_2, p_2+j_2}(S_2/J_{H_2}).
\end{equation*}
Since $\beta_{p_i, p_i+j_i}(S_i/J_{H_i}) =0$ for all $j_i > r_i$, where $r_i = \reg S_i/J_{H_i}$ for $i=1,2$, and $r = r_1 + r_2$, it follows 
\begin{equation*}
\beta_{p,p+r}(S/J_G) =  \beta_{p_1, p_1+r_1}(S_1/J_{H_1})\beta_{p_2, p_2+r_2}(S_2/J_{H_2}).
\end{equation*}
\end{proof}

Let $G$ be a simple connected graph on $[n]$. We recall that if $J_G$ is Cohen-Macaulay, then $p=\projdim S/J_G = n-1$, and it admits an unique extremal Betti number, that is $\Ebeta(S/J_G) = \beta_{p,p+r} (S/J_G)$, where $r = \reg S/J_G$.

\section{Regularity and Cohen-Macaulay type of cones}\label{Sec: cones}

The \textit{cone} of $v$ on $H$, namely $\cone(v,H)$, is the graph with vertices $V(H) \cup \{v\}$ and edges $E(H) \cup \{\{v,w\} \mid w \in V(H)\}$. 

\begin{lemma}\label{sum reg}
Let $G=\cone(v, H_1 \sqcup \dots \sqcup H_s)$, with $s\geq 2$. Then 

\[
\reg S/J_G  = \max \left\lbrace \sum_{i=1}^s \reg S/J_{H_i}, 2\right\rbrace.
\]

\end{lemma}
\begin{proof}
Consider the short exact sequence (\ref{Exact}), with $G=\cone(v, H_1 \sqcup \dots \sqcup H_s)$ and $u=v$, then $G' = K_n$, the complete graph on $[n]$, $G'' = H_1 \sqcup \dots \sqcup H_s$, and $H=K_{n-1}$, where $n = |V(G)|$. Since $G'$ and $H$ are complete graphs, the regularity of $S/J_{G'}$ and $ S/((x_u,y_u)+J_{H})$ is 1. Whereas the regularity of $S/((x_u,y_u)+J_{G''})$ is given by $\reg S/J_{H_1} + \cdots +\reg S/J_{H_s}$. We get the following bound on the regularity of $S/J_G$
\begin{eqnarray*}
  \reg S/J_G\hspace{-0.2cm}&\leq &\hspace{-0.2cm}\max\left\lbrace\reg \frac{S}{J_{G'}},\reg \frac{S}{((x_u, y_u)+J_{G''})}, \reg \frac{S}{((x_u,y_u)+J_{H})} +1\right\rbrace\\
             &= &\hspace{-0.2cm}\max\left\lbrace 1, \sum_{i=1}^s \reg S/J_{H_i}, 2\right\rbrace.
\end{eqnarray*}
Suppose $\sum_{i=1}^s \reg S/J_{H_i} \geq 2$, hence $\reg S/J_G \leq \sum_{i=1}^s \reg S/J_{H_i}$. Since $H_1 \sqcup \dots \sqcup H_s$ is an induced subgraph of $G$, by \cite[Corollary 2.2]{MM} of Matsuda and Murai we have
\[
\reg S/J_{G} \geq \reg S/J_{H_1\sqcup \cdots \sqcup H_s} = \sum_{i=1}^s \reg S/J_{H_i}.
\]
Suppose now $\sum_{i=1}^s \reg S/J_{H_i} < 2$, hence $\reg S/J_G \leq 2$. Since $G$ is not a complete graph, $\reg S/J_G \geq 2$, and the statement follows.

\end{proof}

Observe that it happens $\reg S/J_G  = 2$, for $G=\cone(v, H_1 \sqcup \dots \sqcup H_s)$, with $s \geq 2$, if and only if all the $H_i$ are isolated vertices except for at most two which are complete graphs.
\ \\

We are going to give a description of the Cohen-Macaulay type and some Betti numbers of $S/J_G$ when $S/J_G$ is Cohen-Macaulay, and $G$ is a cone, namely $G=\cone(v,H)$. By \cite[Lemma 3.4]{RR}, to have Cohen-Macaulayness it is necessary that $H$ has exactly two connected components and both are Cohen-Macaulay (see also Corollaries 3.6 and 3.7 and Theorem 3.8 in \cite{RR}).

\begin{proposition}\label{prop: cm-type cone}
Let $G=\cone(v, H_1 \sqcup H_2)$ on $[n]$, with $J_{H_1}$ and $J_{H_2}$ Cohen-Macaulay binomial edge ideals. Then
\[
\CMtype(S/J_G) = n-2 + \CMtype(S/J_{H_1})\CMtype(S/J_{H_2}).
\]
In particular, the unique extremal Betti number of $S/J_G$ is given by
\begin{equation*}
\Ebeta(S/J_G) = 
\begin{cases}
\Ebeta(S_1/J_{H_1}) \Ebeta (S_2/J_{H_2}) & \text{ if } r >2 \\
n-2 +\Ebeta(S_1/J_{H_1}) \Ebeta(S_2/J_{H_2}) & \text{ if } r =2
\end{cases}
\end{equation*}
where $r= \reg S/J_G$. In addition, if $r >2$, it holds 
\begin{equation*}
\beta_{p,p+2} (S/J_G) = n-2.
\end{equation*}

\end{proposition}

\begin{proof}
 Consider the short exact sequence (\ref{Exact}), with $u=v$, then we have $G' = K_{n}$, $G'' = H_1 \sqcup H_2$, and $H=  K_{n-1}$. It holds
\begin{align}
r=\reg S/J_G & \ = \ \max \{\reg S/J_{H_1} + \reg S/J_{H_2}, 2\}, \notag\\
\reg S/((x_u,y_u)+J_{G''}) & \ = \ \reg S/J_{H_1} + \reg S/J_{H_2},\notag \\
\reg S/J_{G'}  & \ = \ \reg S/((x_u,y_u)+J_{H}) = 1, \label{reg1}
\end{align}
and
\[
p=\projdim S/J_G = \projdim S/J_{G'} = \projdim S/((x_u,y_u)+J_{G''}) = n-1,
\]
\[
\projdim S/((x_u,y_u)+J_{H}) = n.
\]
Consider the long exact sequence (\ref{longexact}) with $i=p$. By \eqref{reg1}, we have 
\[
\beta_{p,p+j}(S/J_{G'}) = \beta_{p,p+j}(S/((x_u, y_u)+J_{H})) = 0  \text{ for all } j \geq 2
\]
and 
\[\beta_{p+1,p+1+(j-1)}(S/((x_u,y_u)+J_H)) \neq \text{ only for } j=2.\]

\noindent 
By Lemma \ref{Lemma: beta_p,p+1 e beta_p,p+2} and Lemma \ref{Lemma: Betti numbers of disjoint graphs} (i), it follows that 
\begin{eqnarray*}
 \CMtype(S/J_G) &=& \sum_{j=0}^r \beta_{p,p+j} (S/J_G) = \sum_{j=2}^r \beta_{p,p+j} (S/J_G) \\
 &=& \beta_{p-1,p-2+2}(S/J_H) + \sum_{j=2}^r \beta_{p-2,p-2+j} (S/J_{G''})  \\
 &=& n-2 + \CMtype (S/J_{G''}) \\
 &=& n-2 + \CMtype(S/J_{H_1})\CMtype(S/J_{H_2}).
\end{eqnarray*}

If $r=2$, 
\begin{eqnarray*}
 \CMtype(S/J_G) &=&  \beta_{p,p+2} (S/J_G) \\
 &=& \beta_{p-1,p-2+2}(S/J_H) +  \beta_{p-2,p-2+2} (S/J_{G''})  \\
 &=& n-2 + \Ebeta (S_1/J_{H_1}) \Ebeta (S_2/J_{H_2}),
\end{eqnarray*}
where the last equality follows from Equation (\ref{eq on beta_p,p+2 disjoint graphs}).

If $r >2$, it means that $H_1$ and $H_2$ are not both complete graphs, and then, by Lemma \ref{Lemma: Betti numbers of disjoint graphs} (ii), $\beta_{p-2,p-2+2} (S/J_{G''}) =0$, then $\beta_{p,p+2} (S/J_G) =n-2$, and $\Ebeta(S/J_G) = \Ebeta(S_1/J_{H_1}) \Ebeta(S_2/J_{H_2})$.

\end{proof}

\section{Extremal Betti numbers of some classes of Cohen-Macaulay binomial edge ideals}\label{Sec: bipartite and fan graphs}
We are going to introduce the notation for the family of fan graphs first introduced in \cite{BMS}. 

Let $K_m$ be the complete graph on $[m]$ and $W=\{v_1,\dots,v_s\} \subseteq [m]$. Let $F_{m}^W$ be the graph obtained from $K_m$ by attaching, for every $i=1, \dots, s$, a complete graph $K_{h_i}$ to $K_m$ in such a way $V(K_m) \cap V(K_{h_i}) = \{v_1, \dots, v_i\}$, for some $h_i >i$. We say that the graph $F_m^W$ is obtained by adding a \textit{fan} to $K_m$ on the set $W$. If $h_i = i+1$ for all $i=1, \dots, s$, we say that $F_m^W$ is obtained by adding a \textit{pure fan} to $K_m$ on the set $W$.

Let $W = W_1 \sqcup \cdots \sqcup W_k$ be a non-trivial partition of a subset $W \subseteq [m]$. Let $F_m^{W,k}$ be the graph obtained from $K_m$ by adding a fan to $K_m$ on each set $W_i$, for $i=1, \dots, k$. The graph $F_m^{W,k}$ is called a $k$-fan of $K_m$ on the set $W$. If all the fans are pure, we called it a \textit{k-pure} fan graph of $K_m$ on $W$.

When $k=1$, we write $F_m^W$ instead of $F_m^{W,1}$. Consider the pure fan graph $F_m^W$ on $W=\{v_1, \dots, v_s\}$. We observe that $F_m^W = \cone(v_1, F_{m-1}^{W'} \sqcup \{w\})$, where $W' = W \setminus \{v_1\}$, $w$ is the leaf of $F_m^W$, $\{w,v_1\} \in E(F_m^W)$, and $F_{m-1}^{W'}$ is the pure graph of $K_{n-1}$ on $W'$. \\

Now, we recall the notation used in \cite{BMS} for a family of bipartite graphs. 

For every $m \geq 1$, let $F_m$ be the graph on the vertex set $[2m]$ and with edge set $E(F_m)= \{\{2i, 2j-1\} \mid i=1, \dots, m, j=i, \dots, m\}$. \\

In \cite{BMS}, they prove that if either $G = F_m$ or $G= F_m^{W,k}$, with $m \geq 2$, then $J_{G}$ is Cohen-Macaulay. The regularity of $S/J_G$ has been studied in \cite{JK}, and hold the following results.
\begin{proposition}[\cite{JK}]\label{prop: reg pure fan graph}
Let $G = F_m^{W,k}$ be the $k$-pure fan graph of $K_m$ on $W$, with $m \geq 2$. Then
\[
\reg S/J_G = k+1.
\]
\end{proposition}

\begin{proposition}[\cite{JK}]\label{prop: reg bip}
For every $m\geq 2$, $\reg S/J_{F_m} = 3$.
\end{proposition}

Observe that if $G = F_m^W$ is a pure fan graph, the regularity of $J_G$ is equal to 3 for any $m$ and $W\subseteq [m]$, then all of these graphs belong to the class of graphs studied by Madani and Kiani in \cite{MK1}.

Exploiting Proposition \ref{prop: cm-type cone}, we get hold a formula for the CM-type of any $G = F_m^W$ pure fan graph.

\begin{proposition}\label{Prop: CM-type pure fan graph}
Let $m \geq 2$, and $G = F_m^W$ a pure fan graph, with $|W| \geq 1$. Then
\begin{equation}\label{eq: CM-type pure fan graph}
\mathrm{CM-type}(S/J_G) = \Ebeta(S/J_G)=(m-1)|W|.
\end{equation}
\end{proposition}

\begin{proof}
We use induction on $m$. If $m=2$, $G$ is decomposable into $K_2$ and $K_3$, and it is straightforward to check that (\ref{eq: CM-type pure fan graph}) holds. If $m >2$ and supposing the thesis true for all the pure graphs of $K_{m-1}$, we have $G = \cone(v_1, H_1 \sqcup H_2)$, where $W=\{v_1, \dots, v_s\}$, $H_1=F_{m-1}^{W'}$ is the pure graph of $K_{m-1}$ on $W'$, with $W' = W \setminus \{v_1\}$, $w$ is the leaf of $G$, $\{w,v_1\} \in E(G)$, and $H_2=\{w\}$. By induction hypothesis $\CMtype(S/J_{H_1}) = (m-2)(|W|-1)$,  and $\CMtype(S/J_{H_2})=1$, then using Proposition \ref{prop: cm-type cone}, it follows 
\begin{equation*}
\begin{aligned}
\CMtype(S/J_G) &= |V(G)|-2 + \CMtype(S/J_{H_1})\CMtype(S/J_{H_2}) \\
&= (m+ |W|-2) + (m-2)(|W|-1) = (m-1)|W|.
\end{aligned}
\end{equation*}
Since $|W| \geq 1$, the graph $F_m^W$ is not a complete graph, then $\beta_{p,p+1}(S/J_G) =0$, where $p = \projdim S/J_{G}$. Due to $\reg S/J_G=2$, the $\CMtype(S/J_G)$ coincides with the unique extremal Betti number of $S/J_{G}$, that is $\beta_{p,p+2}$. 
\end{proof}

In the following result we provide the unique extremal Betti number of any $k$-pure fan graph.

\begin{proposition}\label{CM-type Fan}
Let $G = F_m^{W,k}$ be a $k$-pure fan graph, where $m \geq 2$ and $W = W_1 \sqcup \cdots \sqcup W_k \subseteq [m]$ is a non-trivial partition of $W$. Then 
\begin{equation} \label{eq: CM-type k-pure fan graph}
\Ebeta(S/J_G) = (m-1) \prod_{i=1}^k |W_i|.
\end{equation}
\end{proposition}

\begin{proof}
Let $|W_i| = \ell_i$, for $i=1, \dots k$. First of all, we observe that if $\ell_i = 1$ for all $i=1,\dots,k$, that is $W_i = \{v_i\}$, then $G$ is decomposable into $G_1 \cup \cdots \cup  G_{k+1}$, where $G_1=K_m$, $G_j=K_2$ and $G_1 \cap G_j = \{v_j\}$, for all $j=2, \dots, k+1$. This implies 
\[
\Ebeta(S/J_G) = \prod_{j=1}^{k+1} \Ebeta(S/J_{G_j}) = m-1
\]
where the last equality is due to the fact $\Ebeta(S/J_{K_m}) = m-1$ for any complete graph $K_m$, with $m \geq 2$. Without loss of generality, we suppose $\ell_1 \geq 2$. \\
We are ready to prove the statement on induction on $n$, the number of vertices of $G=F_m^{W,k}$, that is $n=m+\sum_{i=1}^k \ell_i$. Let $n=4$, then $G$ is a pure fan graph $F_2^W$, with $|W|=2$, satisfying Proposition \ref{Prop: CM-type pure fan graph} and it holds \eqref{eq: CM-type pure fan graph}. Let $n>4$. Pick $v \in W_1$ such that $\{v,w\} \in E(G)$, with $w$ a leaf of $G$. Consider the short exact sequence (\ref{Exact}), with $u=v$, $G' = F_{m+\ell_1}^{W',k-1}$ the $(k-1)$-pure fan graph of $K_{m+\ell_1}$ on $W' = W_2 \sqcup \cdots \sqcup W_k$, $G'' = F_{m-1}^{W'',k} \sqcup \{w\}$ the disjoint union of the isolated vertex $w$ and the $k$-pure fan graph of $K_{m-1}$ on $W'' = W \setminus \{v\}$, and $H= F_{m+\ell_1-1}^{W',k-1}$. For the quotient rings involved in (\ref{Exact}), from Proposition \ref{prop: reg pure fan graph}, we have
\begin{align*}
r = &\reg S/J_G \ =  \reg S/((x_u,y_u)+J_{G''}) = 1 + k, \\
&\reg S/J_{G'} = \reg S/((x_u,y_u)+J_{H}) =  k.
\end{align*}
As regard the projective dimensions, we have
\begin{align*}
p &= \projdim S/J_G = \projdim S/J_{G'}  = \projdim S/((x_u,y_u)+J_{G''}) \\
&= \projdim S/((x_u,y_u)+J_{H})-1 = m + \sum_{i=1}^k \ell_i -1.
\end{align*}
Fix $i=p$ and $j=r$ in the long exact sequence (\ref{longexact}). The Tor modules $T_{p+1,p+1+(r-1)}(S/((x_u,y_u)+J_H))$ and $T_{p,p+r}(S/((x_u, y_u)+J_{G''}))$ are the only non-zeros. It follows 
\begin{align*}
\beta_{p,p+r}(S/J_G) &= \beta_{p-1,p+r-2}(S/J_{H}) + \beta_{p-2,p+r-2}(S/J_{G''})\\
&= \Ebeta(S/J_H) + \Ebeta(S/J_{F_{m-1}^{W'',k}}).
\end{align*}
Both $F_{m-1}^{W'',k}$ and $H$ fulfil the hypothesis of the proposition and they have less than $n$ vertices, then by induction hypothesis
\begin{align*}
\Ebeta(S/J_{H}) &= (m+\ell_1-2) \prod_{s=2}^k \ell_s, \\
\Ebeta(S/J_{F_{m-1}^{W'',k}}) &= (m-2) (\ell_1-1) \prod_{s=2}^k \ell_s.
\end{align*}
Adding these extremal Betti numbers, the thesis is proved. 
\end{proof}

\begin{proposition}\label{Prop: CM-type bip}
Let $m \geq 2$. The unique extremal Betti number of the bipartite graph $F_m$ is given by 
\[
\Ebeta(S/J_{F_m}) = \sum_{k=1}^{m-1} k^2.
\]
\end{proposition}
\begin{proof}
We use induction on $m$. If $m=2$, then $F_2 = K_2$ and it is well known that 
$\Ebeta(S/J_{F_m})=1$.  
Suppose $m >2$. Consider the short exact sequence (\ref{Exact}), with $G=F_m$ and $u= 2m-1$, with respect to the labelling introduced at the begin of this section. The graphs involved in (\ref{Exact}) are $G' = F_{m+1}^W$, that is the pure fan graph of $K_{m+1}$, with $V(K_{m+1}) = \{u\} \cup \{2i | i=1,\dots,m\}$, on $W=\{2i-1 | i= 1, \dots, m-1\}$, $G'' = F_{m-1} \sqcup \{2m\}$, and the pure fan graph $H=F_m^W$. By Proposition \ref{prop: reg pure fan graph} and Proposition \ref{prop: reg bip}, we have
\begin{align*}
r = &\reg S/J_G = \reg S/((x_u,y_u)+J_{G''}) = 3 \\
&\reg S/J_{G'} = \reg S/((x_u,y_u)+J_{H}) = 2.
\end{align*}
As regards the projective dimension of the quotient rings involved in (\ref{Exact}), it is equal to $p = 2m-1$ for all, except for $S/((x_u,y_u)+J_{H})$ whose projective dimension is $2m$.
Consider the long exact sequence (\ref{longexact}), with $i=p$ and $j=r$.
In view of the above, $T_{p,p+r}(S/J_{G'})$, $T_{p,p+r}(S/((x_u,y_u)+J_H))$, and all the Tor modules on the left of $T_{p+1,p+1+(r-1)}(S/((x_u,y_u)+J_H))$ in (\ref{longexact})  are zero. It follows that 
\[
T_{p,p+r}(S/J_G) \cong T_{p+1,p+1+(r-1)}(S/((x_u,y_u)+J_H)) \oplus T_{p,p+r}(S/((x_u, y_u)+J_{G''})).
\]
Then, using Proposition \ref{CM-type Fan} and induction hypothesis, we obtain 
\begin{eqnarray*}
\beta_{p,p+r}(S/J_G) &=& \beta_{p-1,p+r-2}(S/J_H) + \beta_{p-2,p+r-2}(S/J_{G''})\\ &=& \Ebeta(S/J_H) + \Ebeta(S/J_{G''})\\
&=&(m-1)^2 + \sum_{k=1}^{m-2} k^2 = \sum_{k=1}^{m-1} k^2.
\end{eqnarray*}

\end{proof}

\begin{question}\label{Rem: extremal betti = CM-type}
Based on explicit calculations we believe that for all bipartite graphs $F_m$ and pure fan graphs $F_m^{W,k}$ the unique extremal Betti number coincides with the CM-type, that is $\beta_{p,p+j}(S/J_G) = 0$ for all $j=0, \dots, r-1$, when either $G= F_m$ or $G= F_m^{W,k}$, for $m \geq 2$, and $p= \projdim S/J_G$ and $r= \reg S/J_G$.
\end{question}
\ \\

In the last part of this section, we completely describe the Hilbert-Poincaré series $\HS$ of $S/J_G$, when $G$ is a bipartite graph $F_m$. In particular, we are interested in computing the $h$-vector of $S/J_G$. 

For any graph $G$ on $[n]$, it is well known that 
 \[
 \HS_{S/J_G} (t) = \frac{p(t)}{(1-t)^{2n}} = \frac{h(t)}{(1-t)^d}
 \]
 where $p(t), h(t) \in \ZZ[t]$ and $d$ is the Krull dimension of $S/J_G$. The polynomial $p(t)$ is related to the graded Betti numbers of $S/J_G$ in the following way
 \begin{equation}\label{eq: p(t) with Betti number}
 p(t) = \sum_{i,j}(-1)^i\beta_{i,j}(S/J_G)t^j.
 \end{equation}

\begin{lemma}\label{lemma: the last non negative entry}
Let $G$ be a graph on $[n]$, and suppose $S/J_G$ has an unique extremal Betti number, then the last non negative entry in the $h$-vector is $(-1)^{p+d}\beta_{p,p+r}$, where $p= \projdim S/J_G$ and $r= \reg S/J_G$.
\end{lemma}

\begin{proof}
If $S/J_G$ has an unique Betti number then it is equal to $\beta_{p,p+r}(S/J_G)$. Since $p(t)=h(t)(1-t)^{2n-d}$, then $\lc (p(t)) = (-1)^d \lc(h(t))$, where $\lc$ denotes the leading coefficient of a polynomial. By Equation (\ref{eq: p(t) with Betti number}), the leading coefficient of $p(t)$ is the coefficient of $t^j$ for $j = p+r$. Since $\beta_{i,p+r} = 0$ for all $i < p$, $\lc (p(t)) = (-1)^p \beta_{p,p+r}$, and the thesis follows.
\end{proof}

The degree of $\HS_{S/J_G}(t)$ as a rational function is called \textit{a-invariant}, denoted by $a(S/J_G)$, and it holds
 \[
 a(S/J_G) \leq \reg S/J_G - \depth S/J_G.
 \]
The equality holds if $G$ is Cohen-Macaulay. In this case, $\dim S/J_G = \depth S/J_G$, and then $\deg h(t) = \reg S/J_G$. 

\begin{proposition}
Let $G=F_m$, with $m \geq 2$, then the Hilbert-Poincaré series of $S/J_G$ is given by
\[
\HS_{S/J_G}(t) = \frac{h_0+h_1t+h_2t^2+ h_3t^3}{(1-t)^{2m+1}}
\]
where
\[
h_0 = 1, \qquad h_1= 2m-1, \qquad h_2 = \frac{3m^2-3m}{2}, \; \text{ and } \; h_3 = \sum_{k=1}^{m-1}k^2.
\]
\end{proposition}

\begin{proof}
By Proposition \ref{prop: reg bip}, $\deg h(t) = \reg S/J_G = 3$. Let $\mathrm{in}(J_G) = I_{\Delta}$, for some simplicial complex $\Delta$, where $I_\Delta$ denotes the Stanley-Reisner ideal of $\Delta$. Let $f_i$ be the number of faces of $\Delta$ of dimension $i$ with the convention that
$f_{-1} = 1$. Then 
\begin{equation}\label{eq: h_k}
h_k = \sum_{i=0}^k (-1)^{k-i} \binom{d-i}{k-i} f_{i-1}.
\end{equation}
Exploiting the Equation (\ref{eq: h_k}) we get 
\[
h_1 = f_0 -d = 4m - (2m+1) = 2m-1
\]
To obtain $h_2$ we need first to compute $f_1$, that is the number of edges in $\Delta$: they are all the possible edges, except for those that appear in $(I_{\Delta})_2$, which are the number of edges in $G$. So
\[
f_1 = \binom{4m}{2} - \frac{m(m+1)}{2} = \frac{15m^2-5m}{2}.
\] 
And then we have 
\begin{eqnarray*}
h_2 = \binom{2m+1}{2} f_{-1} - \binom{2m}{1} f_{0} + \binom{2m-1}{0} f_{1} = \frac{3m^2-3m}{2}.
\end{eqnarray*}
By Lemma \ref{lemma: the last non negative entry}, and since $p=2m-1$ and $d=2m+1$, 
\[
h_3 = (-1)^{4m} \beta_{p,p+r}(S/J_G) = \sum_{k=1}^{m-1}k^2
\]
where the last equality follows from Proposition \ref{Prop: CM-type bip}.
\end{proof}

\section{Extremal Betti numbers of Cohen-Macaulay bipartite graphs}\label{Sec: Cohen-Macaulay bipartite graphs}

In \cite{BMS}, the authors prove that, if $G$ is connected and bipartite, then $J_G$ is Cohen-Macaulay if and only if $G$ can be obtained recursively by gluing a finite number of graphs of the form $F_m$ via two operations. Here, we recall the notation introduced in \cite{BMS} for the sake of completeness. \\

\noindent Operation $*$: For $i = 1, 2$, let $G_i$ be a graph with at least one leaf $f_i$. We denote by $G = (G_1, f_1) * (G_2, f_2)$ the graph G obtained by identifying $f_1$ and $f_2$.  \\

\noindent Operation $\circ$: For $i = 1,2$, let $G_i$ be a graph with at least one leaf $f_i$, $v_i$ its neighbour and assume $\deg_{G_i}(v_i) \geq 3$. We denote by  $G = (G_1, f_1) \circ (G_2, f_2)$ the graph G obtained by removing the leaves $f_1, f_2$ from $G_1$ and $G_2$ and by identifying $v_1$ and $v_2$. \\

In $G = (G_1, f_1) \circ (G_2, f_2)$, to refer to the vertex $v$ resulting from the identification of $v_1$ and $v_2$ we write $\{v\} = V(G_1) \cap V(G_2)$. 
For both operations, if it is not important to specify the vertices $f_i$ or it is clear from the context, we simply write $G_1 * G_2$ or $G_1 \circ G_2$. 

\begin{theorem}[\cite{BMS}]
Let $G=F_{m_1} \circ \cdots \circ F_{m_t} \circ F$, where $F$ denotes either $F_{m}$ or a $k$-pure fan graph $F_{m}^{W,k}$, with $t \geq 0$, $m \geq 3$, and $m_i \geq 3$ for all $i=1, \dots, t$. Then $J_G$ is Cohen-Macaulay.
\end{theorem}

\begin{theorem}[\cite{BMS}, \cite{RR}]\label{Theo: bipartite CM}
Let $G$ be a connected bipartite graph. The following properties are equivalent:
\begin{enumerate}
\item[(i)] $J_G$ is Cohen-Macaulay;
\item[(ii)] $G = A_1 *A_2 * \cdots * A_k$, where, for $i=1, \dots, k$, either $A_i = F_m$ or $A_i = F_{m_1} \circ \cdots \circ F_{m_t}$, for some $m \geq 1$ and $m_j \geq 3$.
\end{enumerate}
\end{theorem}

Let $G=G_1* \cdots * G_t$, for $t \geq 1$. Observe that $G$ is decomposable into $G_1 \cup \cdots \cup G_t$, with $G_i \cap G_{i+1} = \{f_i\}$, for $i=1, \dots t-1$, where $f_i$ is the leaf of $G_i$ and $G_{i+1}$ which has been identified in $G_i *G_{i+1}$ and $G_i \cap G_j = \emptyset$, for $1\leq i <j \leq t$. If $G$ is a Cohen-Macaulay bipartite graph, then it admits only one extremal Betti number, and by \cite[Corollary 1.4]{HR}, it holds 
\[
\Ebeta(S/J_G) = \prod_{i=1}^t \Ebeta(S/J_{G_i}).
\]
In light of the above, we will focus on graphs of the form $G= F_{m_1} \circ \cdots \circ F_{m_t}$, with $m_i \geq 3$, $i=1, \dots, t$. Before stating the unique extremal Betti number of $S/J_G$, we recall the results on regularity showed in \cite{JK}.

\begin{proposition}[{\cite{JK}}]
For $m_1, m_2 \geq 3$, let $G = F_{m_1} \circ F$, where either $F = F_{m_2}$ or $F$ is a $k$-pure fan graph $F_{m_2}^{W,k}$, with $W = W_1 \sqcup \cdots \sqcup W_k$ and $\{v\}= V(F_{m_1}) \cap V(F)$. Then
\begin{equation*}
\reg S/J_G = 
\begin{cases}
6 \qquad \; \; \; \quad \text{ if } F = F_{m_2}\\
k+3 \qquad \text{ if } F= F_{m_2}^{W,k} \text{ and } |W_i| = 1 \text{ for all } i\\
k+4 \qquad \text{ if } F= F_{m_2}^{W,k} \text{ and } |W_i| \geq 2 \text{ for some } i \text{ and } v \in W_i\\
\end{cases}
\end{equation*}
\end{proposition}

\begin{proposition}[\cite{JK}]
Let $m_1, \dots, m_t , m\geq 3$ and $t \geq 2$. Consider $G= F_{m_1} \circ \cdots \circ F_{m_t} \circ F$, where $F$ denotes either $F_{m}$ or the $k$-pure fan graph $F_m^{W,k}$ with  $W = W_1 \sqcup \cdots \sqcup W_k$ and $|W_i| \geq 2$ for some $i$. Then 
\[
\reg S/J_G = \reg S/J_{F_{m_1-1}} + \reg S/J_{F_{m_2-2}} + \cdots + \reg S/J_{F_{m_t -2}} + \reg S/J_{F \setminus \{v,f\}}
\]
where $\{v\} = V( F_{m_1} \circ \cdots \circ F_{m_t}) \cap V(F)$, $v \in W_i$ and $f$ is a leaf such that $\{v,f\} \in E(F)$.
\end{proposition}

\begin{lemma}\label{Lemma: CM-type 2 pallini}
Let $m_1, m_2 \geq 3$ and $G = F_{m_1} \circ F$, where $F$ is either $F_{m_2}$ or a $k$-pure fan graph $F_{m_2}^{W,k}$, with $W = W_1 \sqcup \cdots \sqcup W_k$ and $|W_i| \geq 2$ for some $i$. Let $\{v\} = V(F_{m_1}) \cap V(F)$ and suppose $v \in  W_i$. Let $G''$ be as in Set-up \ref{setup}, with $u=v$. Then the unique extremal Betti number of $S/J_G$ is given by
\[
\Ebeta (S/J_G) = \Ebeta (S/J_{G''}).
\]
In particular,
\[
\Ebeta (S/J_G) = 
\begin{cases}
\Ebeta(S/J_{F_{m_1-1}})\Ebeta(S/J_{F_{m_2-1}}) \qquad \text{ if } F = F_{m_2} \\
\Ebeta(S/J_{F_{m_1-1}})\Ebeta(S/J_{F_{m_2-1}^{W',k}}) \qquad \text{ if } F = F_{m_2}^{W,k}
\end{cases}
\]
where $W' = W \setminus \{v\}$.
\end{lemma}

\begin{proof}
Consider the short exact sequence (\ref{Exact}), with $G = F_{m_1} \circ F$ and $u=v$.\\
If $F = F_{m_2}$, then the graphs involved in (\ref{Exact}) are: $G' = F_{m}^{W,2}$, $G'' = F_{m_1 -1} \sqcup F_{m_2-1}$, and $H=F_{m-1}^{W,2}$, where $m=m_1+m_2-1$, $W = W_1 \sqcup W_2$ with $|W_i| = m_i -1$ for $i=1,2$, and $G'$ and $H$ are $2$-pure fan graph. By Proposition \ref{prop: reg pure fan graph} and Proposition \ref{prop: reg bip}, we have the following values for the regularity
\begin{align*}
r = &\reg S/J_G = \reg S/((x_u,y_u)+J_{G''}) = 6\\
&\reg S/J_{G'} = \reg S/((x_u,y_u)+J_{H}) = 3.
\end{align*}
In the matter of projective dimension, it is equal to $p=n-1$ for all the quotient rings involved in (\ref{Exact}), except for $S/((x_u,y_u)+J_{H})$, for which it is $n$. 
Considering the long exact sequence (\ref{longexact}) with $i=p$ and $j=r$, it holds
\[
\beta_{p,p+r} (S/J_G) = \beta_{p,p+r} (S/((x_u,y_u)+J_{G''})) 
\]
and by Lemma \ref{Lemma: Betti numbers of disjoint graphs} (ii) the second part of thesis follows. \\
The case $F = F_{m_2}^{W,k}$ follows by similar arguments. Indeed, suppose $|W_1| \geq 2$ and $v \in W_1$. The graphs involved in (\ref{Exact}) are: $G' = F_{m}^{W',k}$, $G'' = F_{m_1 -1} \sqcup F_{m_2-1}^{W'',k}$, and $H=F_{m-1}^{W',k}$, where $m=m_1+m_2+|W_1| -2$, all the fan graphs are $k$-pure, $W' = W'_1 \sqcup W_2 \sqcup \cdots \sqcup W_k$, with $|W'_1| = m_1 -1$, whereas $ W'' = W \setminus \{v\}$. Fixing $r = \reg S/J_G = \reg S/((x_u,y_u)+J_{G''}) = k+4$, since $\reg S/J_{G'} = \reg S/((x_u,y_u)+J_{H}) = k+1$, and the projective dimension of all the quotient rings involved in (\ref{Exact}) is $p=n-1$, except for $S/((x_u,y_u)+J_{H})$, for which it is $n$, it follows
\[
\beta_{p,p+r} (S/J_G) = \beta_{p,p+r} (S/((x_u,y_u)+J_{G''})) 
\]
and by Lemma \ref{Lemma: Betti numbers of disjoint graphs} (ii) the second part of the thesis follows. 
\end{proof}

\begin{theorem}\label{Theo: betti number t pallini}
Let $t \geq 2$, $m \geq 3$, and $m_i \geq 3$ for all $i=1, \dots, t$. Let $G=F_{m_1} \circ \cdots \circ F_{m_t} \circ F$, where $F$ denotes either $F_{m}$ or a $k$-pure fan graph $F_{m}^{W,k}$  with $W = W_1 \sqcup \cdots \sqcup W_k$. Let $\{v\} = V(F_{m_1} \circ \cdots \circ F_{m_t} ) \cap V(F)$ and, if $F=F_m^{W,k}$, assume $|W_1| \geq 2$ and $v \in  W_1$. Let $G''$ and $H$ be as in Set-up \ref{setup}, with $u=v$. Then the unique extremal Betti number of $S/J_G$ is given by 
\[
\Ebeta(S/J_G)=\Ebeta(S/J_{G''}) +
\begin{cases}
\Ebeta(S/J_H) &\text{if } m_t=3\\
0 &\text{if } m_t>3
\end{cases}
\]
In particular, if $F = F_{m}$, it is given by
\[
\Ebeta ( S/J_G) = 
\Ebeta(S/J_{F_{m_1} \circ \cdots \circ F_{m_t-1}})  \Ebeta(S/J_{F_{m-1}})  + 
\begin{cases}
\Ebeta(S/J_{H}) &\text{if } m_t=3 \\
0  &\text{if } m_t>3
\end{cases}
\]
where $H = F_{m_1} \circ \cdots \circ F_{m_{t-1}} \circ F_{m+m_t-2}^{W',2}$, and $F_{m+m_t-2}^{W',2}$ is a $2$-pure fan graph of $K_{m+m_t-2}$ on $W'=W_1' \sqcup W_2'$, with $|W_1'|=m_t-1$ and $|W_2'|=m-1$.\\
If $F = F_{m}^{W,k}$, it is given by
\[
\Ebeta ( S/J_G) = 
\Ebeta(S/J_{F_{m_1} \circ \cdots \circ F_{m_t-1}}) \Ebeta(S/J_{F_{m-1}^{W'',k}})  + 
\begin{cases}
\Ebeta(S/J_H ) &\text{if } m_t=3 \\
0  &\text{if } m_t>3
\end{cases}
\]
where $W'' = W \setminus \{v\}$, $H=F_{m_1}\circ \cdots \circ F_{m_{t-1}} \circ F_{m'}^{W''',k}$, with $m'=m+m_t+|W_1|-2$, $W''' = W_1'' \sqcup W_2 \sqcup \cdots \sqcup W_k$, and $|W_1''| = m_t -1$.
\end{theorem}

\begin{proof}
If $F=F_m$, we have $G'= F_{m_1} \circ \cdots \circ F_{m_{t-1}} \circ F_{m+m_t-1}^{W',2}$, $G'' = F_{m_1} \circ \cdots \circ F_{m_t-1} \sqcup F_{m-1}$, and $H=F_{m_1} \circ \cdots \circ F_{m_{t-1}} \circ F_{m+m_t-2}^{W',2}$, where $W'=W_1'\sqcup W_2'$, with $|W_1'|=m_t-1$ and $|W_2'|=m-1$. As regard the regularity of these quotient rings, we have
\begin{align*}
r &= \reg S/J_G = \reg S/((x_u,y_u)+J_{G''})\\
&= \reg S/J_{F_{m_1-1}} + \reg S/J_{F_{m_2-2}} + \cdots + \reg S/J_{F_{m_t -2}} + \reg S/J_{F_{m-1}}
\end{align*}
and  both $\reg S/J_{G'}$ and $\reg S/((x_u,y_u)+J_H)$ are equal to
\[
\reg S/J_{F_{m_1-1}} + \reg S/J_{F_{m_2-2}} + \cdots + \reg S/J_{F_{m_{t-1} -2}} + \reg S/J_{F_{m+m_t-1}^{W',2}}.
\]
Since $\reg S/J_{F_{m-1}} = \reg S/J_{F_{m+m_t-1}^{W',2}}=3$, whereas if $m_t=3$, $\reg S/J_{F_{m_t -2}}=1$, otherwise $\reg S/J_{F_{m_t -2}}=3$, it follows that 
\[
\reg S/J_{G'} = \reg S/((x_u,y_u)+J_H) = 
\begin{cases}
r-1 &\text{ if } m_t=3\\
r-3 &\text{ if } m_t>3\\
\end{cases}
\]
For the projective dimensions, we have
\begin{eqnarray*}
p &=& \projdim S/J_G = \projdim S/((x_u,y_u)+J_{G''}) \\
&=& \projdim S/J_{G'} = \projdim S/((x_u,y_u)+J_{H}) -1= n-1.
\end{eqnarray*}
Passing through the long exact sequence (\ref{longexact}) of Tor modules, we obtain, if $m_t =3$
\[
\beta_{p,p+r} (S/J_G) = \beta_{p,p+r} (S/((x_u,y_u)+J_{G''})) + \beta_{p+1,(p+1)+(r-1)}(S/((x_u,y_u)+J_H))
\]
and, if $m_t>3$
\[
\beta_{p,p+r} (S/J_G) = \beta_{p,p+r} (S/((x_u,y_u)+J_{G''})).
\]
The case $F=F_{m}^{W,k}$ follows by similar arguments. Indeed, the involved graphs are: $G'=F_{m_1}\circ \cdots \circ F_{m_{t-1}} \circ F_{m'}^{W''',k}$, $G''= F_{m_1}\circ \cdots \circ F_{m_t-1} \sqcup F_{m-1}^{W'',k}$, and $H=F_{m_1}\circ \cdots \circ F_{m_{t-1}} \circ F_{m'-1}^{W''',k}$, where all the fan graphs are $k$-pure, $W'' = W \setminus \{v\}$, $m'=m+m_t+|W_1|-1$, $W''' = W_1'' \sqcup W_2 \sqcup \cdots \sqcup W_k$, and $|W_1''| = m_t -1$. Fixing $r= \reg S/J_G$, we get $\reg S/((x_u,y_u)+J_{G''}) =r$, whereas 
\[
\reg S/J_{G'} = \reg S/((x_u,y_u)+J_H) = 
\begin{cases}
r-1 &\text{ if } m_t=3\\
r-3 &\text{ if } m_t>3\\
\end{cases}
\]
The projective dimension of all the quotient rings involved is $p=n-1$, except for $S/((x_u,y_u)+J_H)$, for which it is $n$. Passing through the long exact sequence (\ref{longexact}) of Tor modules, it follows the thesis. 
\end{proof}

\begin{corollary}
Let $t \geq 2$, $m,m_1 \geq 3$, and $m_i \geq 4$ for all $i=2, \dots, t$. Let $G=F_{m_1} \circ \cdots \circ F_{m_t} \circ F$, where $F$ denotes either $F_{m}$ or a $k$-pure fan graph $F_{m}^{W,k}$  with $W = W_1 \sqcup \cdots \sqcup W_k$. Let $\{v\} = V(F_{m_1} \circ \cdots \circ F_{m_t} ) \cap V(F)$ and, when $F=F_m^{W,k}$, assume $|W_1| \geq 2$ and $v \in  W_1$. Then the unique extremal Betti number of $S/J_G$ is given by 
\[
\Ebeta(S/J_G)=
\begin{cases}
\Ebeta(S/J_{F_{m_1-1}})\prod_{i=2}^t \Ebeta(S/J_{F_{m_i-2}})\Ebeta(S/J_{F_{m-1}})&\text{ if } F=F_m\\
\Ebeta(S/J_{F_{m_1-1}})\prod_{i=2}^t \Ebeta(S/J_{F_{m_i-2}})\Ebeta(S/J_{F_{m-1}^{W',k}}) &\text{ if } F=F_m^{W,k}
\end{cases}
\]
where $W' = W\setminus \{v\}$.

\end{corollary}

\begin{proof}
By Theorem \ref{Theo: betti number t pallini} and by hypothesis on the $m_i$'s, we get
\[
\Ebeta(S/J_G) = \Ebeta(S/J_{F_{m_1}} \circ \cdots \circ F_{m_{t-1}})  \Ebeta(S/J_{F_{m-1}}).
\]
Repeating the same argument for computing the extremal Betti number of $S/J_{F_{m_1} \circ \cdots \circ F_{m_t-1} }$, and by Lemma \ref{Lemma: CM-type 2 pallini}, we have done. 
\end{proof}

\begin{remark}
Contrary to what we believe for bipartite graphs $F_m$ and $k$-pure fan graphs $F_m^{W,k}$ (see Question \ref{Rem: extremal betti = CM-type}), in general for a Cohen-Macaulay bipartite graph $G=F_{m_1} \circ \cdots \circ F_{m_t}$, with $t \geq 2$, the unique extremal Betti number of $S/J_G$ does not coincide with the Cohen-Macaulay type  of $S/J_G$, for example for $G=F_4 \circ F_3$, we have $5 = \Ebeta(S/J_G) \neq \CMtype(S/J_G) =29 $.
\end{remark}

\end{document}